\documentclass[draft,10pt]{article}
\usepackage{amsfonts}
\usepackage{amssymb}
\usepackage{amsbsy}
\usepackage{amsthm}
\usepackage{amsmath}
\usepackage{amscd}
\usepackage{color}

\newtheorem{thm}{Theorem}[section]
\newtheorem{cor}[thm]{Corollary}
\newtheorem{lem}[thm]{Lemma}
\newtheorem{prop}[thm]{Proposition}
\theoremstyle{definition}
\newtheorem{defn}[thm]{Definition}
\theoremstyle{remark}
\newtheorem{rem}[thm]{Remark}
\newtheorem{ex}[thm]{\bf Example}

\newcommand{\bt}{\begin{thm}}
\newcommand{\et}{\end{thm}}
\newcommand{\bc}{\begin{cor}}
\newcommand{\ec}{\end{cor}}
\newcommand{\bl}{\begin{lem}}
\newcommand{\el}{\end{lem}}
\newcommand{\bp}{\begin{prop}}
\newcommand{\ep}{\end{prop}}
\newcommand{\bd}{\begin{defn}}
\newcommand{\ed}{\end{defn}}
\newcommand{\br}{\begin{rem}}
\newcommand{\er}{\end{rem}}
\newcommand{\bpr}{\begin{proof}}
\newcommand{\epr}{\end{proof}}
\newcommand{\bex}{\begin{ex}}
\newcommand{\eex}{\end{ex}}
\newcommand{\bcd}{\begin{CD}}
\newcommand{\ecd}{\end{CD}}

\newcommand{\bi}{\begin{itemize}}
\newcommand{\ei}{\end{itemize}}
\newcommand{\be}{\begin{enumerate}}
\newcommand{\ee}{\end{enumerate}}
\newcommand{\ba}{\begin{array}}
\newcommand{\ea}{\end{array}}
\newcommand{\beq}{\begin{equation}}
\newcommand{\eeq}{\end{equation}}
\newcommand{\beqa}{\begin{eqnarray}}
\newcommand{\eeqa}{\end{eqnarray}}

\newcommand{\eca}{\end{cases}}
\newcommand{\bal}{\begin{aligned}}
\newcommand{\eal}{\end{aligned}}

\newcommand{\ds}{\displaystyle}

\newcommand{\Z}{{\mathbb Z}}
\newcommand{\R}{{\mathbb R}}
\newcommand{\C}{{\mathbb C}}
\newcommand{\T}{{\mathbb T}}
\newcommand{\I}{{\mathbb I}}

\newcommand{\PP}{{\mathbb P}}

\newcommand{\spn}{{\operatorname{span}}}
\newcommand{\Hom}{{\operatorname{Hom}}}
\newcommand{\leb}{{\operatorname{leb\;}}}
\newcommand{\re}{{\operatorname{Re\,}}}
\newcommand{\im}{{\operatorname{Im\,}}}

\newcommand{\<}{\langle}
\renewcommand{\>}{\rangle}
\newcommand{\uk}{|\kern-3.3pt\uparrow\>}
\newcommand{\dk}{|\kern-3.3pt\downarrow\>}
\newcommand{\ub}{\<\uparrow\kern-3.3pt|}
\newcommand{\db}{\<\downarrow\kern-3.3pt|}

\begin{document}
\title{\bf Linear spectral transformations of  Carath\'eodory functions}
\author{M.J. Cantero, L. Moral, L. Vel\'azquez
\footnote{The work of the authors have been supported in part by the research project MTM2011-28952-C02-01 from the Ministry of Science and Innovation of Spain and the European Regional Development Fund (ERDF)   and by Project E-64 of
Diputaci\'{o}n General de Arag\'{o}n (Spain).}}
\date{\small Departamento de Matem\'atica Aplicada and IUMA. Universidad de Zaragoza. Spain}
\maketitle

\begin{abstract} In this paper we present some recent results
concerning linear spectral transformations of Carath\'eodory
functions. More precisely, given two Carath\'eodory functions
related by a linear spectral transformation, we study the relation
between the corresponding moment functionals and, in the positive
definite case, the relation between the measures.

We will see that rational modifications of functionals are included
in the linear spectral transformations. However, we will show that
there exist a huge class of linear spectral transformations which
are not given by rational modifications of  functionals. Indeed, we
will characterize those linear spectral transformation which come
from a rational modification.

In the general case we will discuss the relation between the
functionals involved in a linear spectral transformation, which
allows us to identify  the difficulties to connect the related
functionals.

Actually, several examples will show how amazing can be the
relationships between the moment functionals associated with a
linear spectral transformation.

\end{abstract}

\noindent{\it Keywords and phrases}: Hermitian functionals, rational
modifications of functionals,
 linear spectral transformations, Carath\'eodory functions.
\medskip

\noindent{\it (2000) AMS Mathematics Subject Classification}: 42C05.

\section{Introduction} \label{INT}

\medskip

  The study of modifications of hermitian linear functionals has been
considered by several authors. Among others, the contributions
\cite{ACMV,BuMa04,CaMoVe,CaGaMa} are remarkable. In some cases the
relation between the  Carath\'eodory functions is known. One of the
cases that has been studied more profusely is the situation in which
the Carath\'eodory functions are related by a rational
transformation, (see \cite{PeSt96}). We will study a special class
of these transformation, the so-called linear spectral
transformations. We will perform the analysis in the general case of
(non necessarily positive-definite, neither quasi-definite)
hermitian functionals.

\medskip

Before discussing the results we will introduce some conventions. We
will use the notation $\PP=\C[z]$, $\PP_*=\C[z^{-1}]$,
$\Lambda=\PP\cup \PP_*$, $\PP_n= \spn\{1,z,\dots,z^n\}$ and
$\Lambda_{p,q}=\spn\{z^p,\dots,z^q\}$, $p\leq q$, $p,q\in\Z$. If
$p=-\infty$ or $q=\infty$ we will not understand $\Lambda_{p,q}$ as
a set of finite linear combinations of powers of $z$,
but as a complex linear space of formal series.
We will work whit functionals $u\in \Lambda' = \Hom (\Lambda, \C)$.
For any $u\in\Lambda'$ and $L\in \Lambda$ we define
 $uL\in \Lambda'$ by $uL[f]=u[Lf]$. If $\mu_n=u[z^n]$, $n \in\Z$, are
 the moments of $u$, we will say that the functional $u$ is
hermitian if $\mu_{-n}=\overline{\mu}_n$. In this case $\mu_0\in\R$.
We use the notation
  $\mathcal{H} = \{u \in  \Lambda' \;|\;u\;\text{is Hermitian}\}$. A
  particular case of hermitian functionals are those defined by a
  measure $\mu$ on the unit circle $\T:=\{z\in\C \;|\; |z|=1\}$ via
  $u[f]=\int_{\T} f(z) \, d\mu(z)$.
  The functional given by a Dirac delta at a point
  $\alpha$ will be denoted by $\delta_{\alpha}$, while the functional defined by the Lebesgue
  measure on the unit circle will be denoted by leb.

\medskip

 For  $u \in \mathcal{H}$ we define the Carath\'eodory formal series
  (CS)
  $$F(z)= \ds\mu_0 + 2\sum_{n\geq 1} \mu_{-n} z^n \in \Lambda_{0, \infty}.$$
 If $F$ is summable in a certain domain $\Omega\subset\C$, we will refer to
 $F$ as a Carath\'eodory function in $\Omega$.
   The operators $_*$ and
  $^{*_p}$, $p\in\Z$, act on a formal series $H$ --and, in particular, on
  any Laurent polynomial-- as $H_*(z)=\overline{H(1/\overline{z})}$ and $H^{*_p}(z) =
  z^p H_*(z)$. If $P$ is a polynomial of degree $p$ we
  write $P^{*_p}=P^*$.
  Given $u\in \Lambda'$, we define the  Laurent formal series (LS)
  $$
  \mathcal{L}[u](z)=\sum_{-\infty}^\infty \mu_{-n} z^n \in \Lambda_{-\infty,
  \infty},
  $$
which characterizes completely the functional $u$, i.e., there is a
one-to-one correspondence
$$
\begin{aligned}
\mathcal{L}:\Lambda' &\longrightarrow \{LS\}. \\ u &\longmapsto
\mathcal{L}[u]
\end{aligned}
$$
If the functional $u$ is hermitian, its LS  and CS are related by
 $$   \mathcal{L}[u](z)={F(z)+F_*(z)\over 2}.
 $$
It is immediate to check that
  $$
\mathcal{L}[u + v] = \mathcal{L}[u] + \mathcal{L}[v],  \qquad
\mathcal{L}[u L] = \mathcal{L}[u] L_*, \qquad \forall u\in
\Lambda', L\in \Lambda.
$$

\medskip

\section{Linear spectral transformations and rational modifications} \label{LSTRM}
  \bigskip

\noindent In what  follows, $u$ and $v$ will denote two  hermitian
linear functionals with Carath\'eodory series $F$ and $G$
respectively.
 \vskip 0.25cm
\begin{defn}
  We will say that $u$ and $v$ are related by a  linear spectral transformation (LST) if there
  exist  $L,M \in \Lambda_0:=\Lambda\setminus \{0\}$ and $C\in \Lambda$ such that $FL=GM+C$.
 \end{defn}
\vskip 0.25cm
\begin{defn}
 We will say that $u$ and $v$ are related by a rational modification
 (RM) if there exist $L,M \in \Lambda_0$ such that $uL=vM$.
   \end{defn}

   \medskip

  We are interested in studying the possible connections between two
hermitian linear functionals $u$
 and $v$ whose corresponding Carath\'eodory functions are related by a LST.
 In particular we will clarify the relation between RM and LST. Our first result is the
 following one.
 \begin{prop}\label{FR} If the linear functionals $u$, $v$  are related by a {\rm RM},
 then the corresponding CS are linked by a {\rm LST}, i.e., {\rm RM} $\Rightarrow$ {\rm LST.}
 \end{prop}

 \begin{proof} Let $u,v\in \Lambda'$ such that $uL_*=vM_*$, with
 $L,M\in \Lambda_{p,q}$  ($p\leq q$, finite). Equivalently,
 $$
 (F+F_*)L=(G+G_*)M.
 $$
Taking $C=FL-GM=G_*M-F_*L$, then $FL-GM\in \Lambda_{p,\infty}$,
 $G_*M-F_*L\in \Lambda_{-\infty,q}$. In other words, $C\in
 \Lambda_{p,q}$ satisfies $FL=GM+C$.
 \end{proof}

\medskip

\begin{rem}
 The converse of Proposition \ref{FR} is not true in general.
 To see this, it is enough
to consider  $v = \delta_1$,  whose  Carath\'eodory function is
$G(z)=(1+z)/(1-z)$, $z\in\C\setminus\{1\}$. The LST given by
$F(z)=G(z)(1-z^2)$ provides a Carath\'eodory function $F$ without
poles and, consequently, the corresponding orthogonality measure
$\mu$ is absolutely continuous. Besides, $\mu$ is supported on the
whole unit circle. Hence, there is no RM between $u$ and $v$
because, if two functionals defined by measures are related by a RM,
the limit points of the support of both measures must coincide.
\end{rem}

\medskip

 Although not every LST comes from a RM, Proposition \ref{FR} states
 that RM constitute an important subclass  of LST. Our first aim is to
 perform a detailed analysis  of the RM subclass. The study  of the
 transformations in $\{\rm LST\}  \setminus \{\rm RM\}$ will be the objective of
 Section 4.

\medskip
First of all we should note that RM and LST are both equivalence
relations. Proposition \ref{FR} states that RM is a finer
equivalence relation than LST. To study properly  the RM relation it
is convenient to separate an equivalence class to avoid certain
ambiguities.

 \medskip

\begin{defn}\label{Delta} We define the class $\Delta$ of hermitian linear
functionals by
$$
\Delta=\{u \in \mathcal{H} \; \mid \; \exists L\in \Lambda_0 \;{\rm
s.t.}\; uL=0\}.
$$
\end{defn}

 Definition \ref{Delta} means that $\Delta$ is the equivalence
 class of the null functional with respect to the equivalence
 relation RM. The referred ambiguity comes from the fact that
 every pair $u,v\in\Delta$ satisfies $uL=vM=0$ for some
$L,M \in \Lambda$. Therefore, the equality $uL=vMN$ is also true for
all $N\in\Lambda$. Thus, the correspondence $L\leftrightarrow M$ is
not biunivocal in $\Delta$. To avoid these ambiguities we will work
with RM in $\mathcal{H}\setminus \Delta$.

\medskip

 The condition $u\in \Delta$ is equivalent to $uQ=0$  for some $Q\in\PP$, which means that
 the functional $u$ is a linear combination of $\delta's$ and their derivatives
supported on the non-zero roots of $Q$. Since the functional $u$ is
hermitian, these $\delta's$ are supported on symmetric points
$\alpha$, $1/\overline{\alpha}$ with conjugated masses.
 Thus, we can choose the minimal polynomial $Q$ such that $uQ=0$ satisfying $Q=Q^*$.
 The minimality of $Q$ means that $uQ_1=0$ implies that $Q$
divides $Q_1$. In such a case the CS  $F$  of $u$  is summable as a
meromorphic function in $\C$, $F=P/Q$, $P$ being a polynomial with
$\deg P\leq \deg Q=q$.

 In terms of LST,
 $$(F+F_*)Q_*=0.$$
Thus,
$$z^{-q}FQ+F_*Q_*=0 \qquad {\rm and}\qquad  z^{-q}P+P_*=0,$$
 i.e.,  $$
P=-z^qP_*=-P^{*_q}.$$
Summarizing, for any $u \in \Delta$, $F$ is
summable as a rational function
$$
F={P\over Q}, \qquad Q=Q^*, \qquad P^{*_{\deg Q}} = -P.
$$

Actually, the above property
characterizes the functionals  belonging to
 $\Delta$. In fact, from
 $$FQ=P$$
 and
 $$
 F_*Q_*=P_*=z^{-q}P^{*_q}=-z^{-q}P
 $$
  we obtain $$(F+F_*)Q=0,$$
i.e.,  $$uQ=0.$$
 Note that $P(0)/Q(0)\in\R$  due to  the hermitian
character of $u$.

\medskip

  Besides, if  $Q$ is minimal, $\gcd (P,Q)=1$.
  Indeed, if $S$ is a common divisor of  $P$ and $Q$, $S^*$ is also common divisor
 of $P$ and $Q$. Then, $P$ and $Q$ are divisible by some $T\in\PP$ with $T=T^*$. The
 polynomials $P_1=P/T$ and $Q_1=Q/T$   satisfy $Q_1^*=Q_1$,
 $P_1^{*_{\deg Q_1}}=-P_1$ and $F=P_1/Q_1$. Thus, $uQ_1=0$,  which
 contradicts the minimality of $Q$.

 \medskip

 In what follows we will work with RM in $\mathcal{H}\setminus
 \Delta$ to avoid the ambiguities in the relation $L\leftrightarrow
 M$ that appear in the class $\Delta$.

 \medskip

\begin{prop} Let $u,v\in \mathcal{H}$. Then, if $uL=vM$, the
correspondence $L\leftrightarrow M$ is biunivocal iff
$u,v\not\in\Delta$.
\end{prop}

\begin{proof} We have shown that the correspondence $L\leftrightarrow M$ is not
biunivocal if $u$ and $v$  lie in the equivalence class
$\Delta$. On the contrary, if $u,v\in\mathcal{H}\setminus\Delta$, then $vM_1=uL=vM$ implies $v(M-M_1)=0$, thus $M=M_1$.
 \end{proof}

\medskip

\section{Minimality of RM and LST} \label{MINRMLST}
\medskip

 For RM in $\mathcal{H}\setminus \Delta$ the correspondence $L \leftrightarrow   M$ is
 biunivocal, but the pair $(L,M)$ is not unique. Indeed,  $uL=vM$ implies
 $uLN=vMN$, $\forall N\in\Lambda$. Our aim is to choose the simplest pair $(L,M)$ satisfying $uL=vM$
 for each $u, v$ in the same RM equivalence class of $\mathcal{H}\setminus\Delta$. This means
 to cancel the common factors of $L$ and $M$ when it is possible.
 However, this cancelation is  not always viable as we show in the following example.

\medskip

 \bex  Let $u=\leb$ and  $v=\leb +\delta_1$. Obviously, the
relations
$$
u(z^2-1)=v(z^2-1)\qquad {\rm and}\qquad u(z-1)=v(z-1)
$$
are satisfied.  From the first relation to the second one,  a common
factor is simplified. However, the second relation is not reducible
due to the presence of $\delta_1$ in $v$. \eex
\medskip

  We are interested in minimal expressions which are not reducible by simplifying common
factors of $L$ and $M$. When $u,v\in \Delta$, the minimal
expressions are not unique: $\delta_1(z-1)=\delta_{-1}(z+1)$ and
$\delta_1(z^2-1)=\delta_{-1}(z+1)$ are both not reducible. In
contrast, we will see that there is an essentially  unique minimal
expression of a RM if $u,v \not\in \Delta$.

\medskip

For this purpose it is convenient to consider the set of all pairs
$(L,M)$ behind a given RM.

\medskip

\begin{defn} Given $u,v \in \mathcal{H}$   we define
$$
 \I\;\;=\;\;\I(u,v)\;\;= \{(L,M)\in \Lambda_0\times\Lambda_0 \hskip 3pt|\hskip 3pt uL=vM
 \}.
$$
\end{defn}

\medskip

 Note that ${\rm U}=\{\alpha z^k \hskip 3pt|\hskip 3pt \alpha\in\C^*, k\in\Z\}$ is
 the group of units of the ring $\Lambda$. Thus, we can define
 a partial order relation in $\I$ by
 $$(L_0,M_0)\leq (L_1,M_1) \quad \Longleftrightarrow \quad N(L_0,M_0)=(L_1,M_1), \quad N\in\Lambda,$$
 and
 $$
 (L_0,M_0)\equiv (L_1,M_1) \quad \Longleftrightarrow \quad N(L_0,M_0)=(L_1,M_1), \quad N\in {\rm U}.
 $$

  Whit respect to this order relation, $\I$ is an inductive set.
 Since $\I$ is  bounded from below by $(0,0)$, $\I$ has at least a
 minimal element. In what follows we will suppose that $\I\not=\emptyset$, i.e., $u,v$  are related by a RM.

 \medskip

 Our first result about minimality  uses the notion of gcd in $\Lambda$, which is unique up to
 factors in  $\rm{U}$. Indeed the division algorithm translates from
 $\PP$ to  $\Lambda = \rm {U} \PP$, and the gcd becomes unique in $\Lambda/{\rm U}\cong
 \PP/\C^*$.

\medskip \bt Let $u,v \in\mathcal{H}\setminus \Delta$ be related by a RM. Then, $\I= \I(u,v)$
has a minimal element $(L,M)$ which is unique up to units of
$\Lambda$. Besides, $\I=\Lambda_0(L,M):=(\Lambda_0L,\Lambda_0M)$.\et

\begin{proof}
 Let $(L_0,M_0) =(L,M)$ be a minimal element of $\I$  and $(L_1,M_1)\in
 \I$.
 Applying the Euclidean algorithm in $\Lambda$,
 $$
 L_k=L_{k+1}Q_{k+1} + L_{k+2}, \qquad k=0,1,\dots,m-2,
 $$
 with $Q_{k+1}\in \Lambda$ and $L_m\in\gcd(L_0,L_1)$. Therefore,
 $$
 uL_k=uL_{k+1}Q_{k+1} + uL_{k+2} = vM_k = vM_{k+1}Q_{k+1} + u
 L_{k+2}.
 $$
 Hence
 $$
 u L_{k+2} = v(M_k-M_{k+1} Q_{k+1}) = vM_{k+2}.
 $$
 After $m-2$ steps, we arrive at $uL_m=vM_m$. However, for $i=0,1$, $L_i=P_iL_m$
 with $P_i\in\Lambda$, thus
 $
 uL_i=uP_iL_m=vP_iM_m$
 and  $(L_i,M_i)=P_i(L_m,M_m)$.

 Since $(L_0,M_0)$ is minimal, then $P_0\in U$. Hence $(L_1,M_1)=P_1P_0^{-1}(L_0,M_0)$ and $\I=\Lambda (L,M)$.
 Besides, $(L_1,M_1)$ minimal leads to $P_1\in {\rm U}$, so $(L_1,M_1)\equiv
 (L_0,M_0)$.
 \end{proof}

\medskip

 The minimal representation of RM has an additional advantage.

 \medskip

\bc \label{MINIM}If $u,v\in \mathcal{H}\setminus \Delta$ and $uL=vM$
is minimal, then  $(L_*,M_*)\equiv (L,M)$, i.e., there exist
$\alpha\in\C^*$ and $p\in\Z$ such that $(L^{*_p},M^{*_p}) =
\alpha(L,M)$.\ec

\begin{proof}Since $uL=vM$ and  $u,v\in \mathcal{H}$  we have that $uL_*=vM_*$. The
minimality ensures the existence of $N\in\Lambda$ such that
$(L_*,M_*)=N(L,M)$, therefore $(L,M)=N_*(L_*,M_*)$ and the
minimality of $(L,M)$ implies $N\in {\rm U}$.
 \end{proof}

\medskip

   Multiplying  by a suitable factor in $\C^*$,  the minimal relation
$uL=vM$, we can find a minimal pair such that $(L,M)=(L^{*_p},
M^{*_p})$ for some $p\in\Z$. This kind of relation also holds for
any other (non-minimal) pair $(L_1,M_1) = N(L,M)$ as far as $N_*=z^k N$ for some $k\in\Z$.
A pair satisfying this condition will be called a
symmetric pair. Due to Corollary \ref{MINIM},  the existence of
minimal representations ensures the existence of symmetric pairs.

\medskip

The use of symmetric pairs allows us to  characterize easily those
LST which represent  RM.

\bc\label{CHAR} Let $u,v\in \mathcal{H}\setminus \Delta$, with $F,G$
 the corresponding {\rm CS}, and $L,M\in\Lambda$ such that
$(L^{*_p}, M^{*_p})=(L,M)$. Then, $uL=vM$ iff $FL=GM+C$, where
$C^{*_p}=-C$.\ec

\begin{proof}
From Proposition \ref{FR},   $C=FL-GM=G_*M-F_*L$, hence
$$
C^{*_p} = z^p C_* = z^p(F_*L_*-G_*M)=F_*L-G_*M=-C.
$$
Conversely, suppose $FL=GM+C$, with $(L^{*_p}, M^{*_p}) = (L,M)$ and
$C^{*_p}=-C$. Then,
$$
FL=GM+C,\qquad F_*L^{*_p} = GM^{*_p} + C^{*_p},
$$
and,
$$(F+F_*)L=(G+G_*)M. $$ The identification  of  the LS  $(F+F_*)$  and $(G+G_*)$
  with $u$ and $v$, respectively,  gives $uL_*=vM_*$
or, equivalently, $uL=vM$.
\end{proof}
\medskip

Example 3.2 shows that it is not always possible to simplify the
common factors in a RM. The corresponding LST for this example,
$F(z^2-1) = G(z^2-1)+(z+1)^2$, suggests that the common factors of $L$, $M$
that we can eliminate in a relation $uL=vM$,
are only those which are also common to $C$ in the corresponding LST $FL=GM+C$.
This conjecture will be proved in the
following theorem which characterizes the minimal representation of
a RM in terms of the corresponding LST.

\medskip

\bt\label{CHART} Let $u,v\in \mathcal{H}\setminus \Delta$ and
$L,M\in \Lambda$ such that $(L^{*_p}, M^{*_p})=(L,M)$. Then, $uL=vM$
is minimal iff the corresponding {\rm LST} $FL=GM+C$ is such that
$1\in\gcd(L,M,C)$.\et

\begin{proof} Corollary \ref{CHAR} ensures the existence of
$C\in\Lambda$ such that $FL=GM+C$, and $C^{*_p}=-C$. This fact
together  with the symmetry of the pair $(L,M)$ ensures the existence
of a representative $P\in\gcd (L,M,C)$ such that $P\in\PP$ and
$P^*=P$.  Denoting $(L_1,M_1,C_1)=(L,M,C)/P$, we have $(L_1^{*_q},
M_1^{*_q}, C_1^{*_q})=(L_1,M_1,-C_1)$ for some $q\in\Z$. Corollary
\ref{CHAR} implies $uL_1= vM_1$. Therefore, the minimality of
$uL=vM$ requires $\deg P=0$.

  Conversely, if $L,M,C$ are coprime and $uL=vM$
 is not minimal, there exists $P\in\PP$ dividing $L,M$ with $P^*=P$. Let $(L_1,M_1)=(L,M)/P$.
  Then,  $(L_1^{*_q},M_1^{*_q})=
 (L_1,M _1)$ for some $q\in\Z$, and $uL_1 =vM_1$. Thus, Corollary \ref{CHAR}
 ensures the existence of $C_1\in \Lambda$ such that $FL_1=GM_1+C_1$ and,
 consequently, $C_1=C/P$. Since $L,M,C$ are coprime, $\deg P=0$.
\end{proof}

\medskip

  Theorem  \ref{CHART} shows how the LST help us to know when a RM is
expressed in a minimal way. In what follows  we study the minimality
of the LST, i.e., $$ FL=GM+C, \qquad 1\in\gcd(L,M,C).$$

\medskip

 As in the case of the equivalence relation RM, there is an
 equivalence class for LST where certain ambiguities take place: the
 LST equivalence class of the null functional.

 \medskip

\begin{defn} We define the class of rational functionals as the set
$$
{\rm Rat}:=\{ u \; \in \mathcal{H}\;\;|\;\; \exists\;
L\in\Lambda_0,\; N\in\Lambda\;{\rm s.t.}\;\; FL=N\}.
$$
\end{defn}

\noindent Obviously, Rat are the set of functionals whose CS are
summable  as rational functions. Hence $\Delta \subset {\rm Rat}$.
In the following proposition we describe the functionals belonging
to the Rat class.

\medskip

\begin{prop} Let $u\in\mathcal{H}$  and $F$ its  Carath\'eodory
function. Then, $u\in {\rm Rat}$ iff $\exists (L,M)\in
\Lambda_0\times \Lambda$ such that $uL=\leb M$.
\end{prop}

\begin{proof} The functionals $ u\in \Delta$ correspond to  $M=0$.
 If $u\in {\rm Rat}\setminus \Delta$, there exist $P, Q \in \PP$
 such that $FQ=P$. Since $u$ is hermitian, the poles of its Carath\'{e}odory function must be
symmetric whit respect to $\T$. Thus, we can choose $Q=Q^*$, $\deg
Q=q$. Consequently, from
$$
FQ={P+P^{*_q}\over 2} + {P-P^{*_q}\over 2} = M + C,
$$
where $\ds M=\frac{P+P^{*_q}}{2}=M^{*_q}$ and $\ds
C=\frac{P-P^{*_q}}{2}=-C^{*_q}$, and Corollary \ref {CHAR}, we find
that $uQ=\leb M$.

 Conversely, if $uL=\leb M$, the hermiticity of $u$ allows us
 to assume without loss of generality that $L=L^{*_p}$, $M=M^{*_p}$, for
 some $p\in\Z$. Corollary \ref{CHAR} ensures the existence of $C\in\Lambda$ such
that $C=-C^{*_p}$ and $FL=M+C=N$.
\end{proof}

 \medskip

  Given the  LST  $FL=GM+C$, the equality $FLN=GMN+CN$ is also a
  LST for all $N\in \Lambda$. In contrast to RM,
 the  LST  are  reducible to expressions where $L,M,C$ are
 coprime, which will be the the minimal  LST. As a last question for this section,
 we ask ourselves when a minimal  LST  is unique (up to units of $\Lambda$)
 for a pair of hermitian functionals $u,v$. The answer is given by
 the next proposition.

\begin{prop} \label{CHARMIN}Let $u,v\in\mathcal{H}$ and   $F,G$ their respective {\rm CS}.
Then, $u,v\not\in {\rm Rat}$  iff the minimal {\rm LST}, $FL=GM+C$,
is unique up to units of $\Lambda$.
\end{prop}

\begin{proof} Let $FL=GM+C$ and $FL_1=GM_1+C_1$ be minimal, such that
$(L_1,M_1,C_1)\not\in {\rm U}(L,M,C)$. Then,
$$
FLL_1=GML_1+CL_1=GM_1L+C_1L
$$
implies
$$
G(ML_1-LM_1) = C_1L-CL_1.
$$
If  $v\not\in{\rm Rat}$, both sides of the above equation will be
zero, which implies the proportionality between $(L,M,C)$ and
$(L_1,M_1,C_1)$.

Conversely, let $u,v\in {\rm Rat}$,  $FQ=P$, $GS=R$, with
$P,Q,R,S\in\PP$. Obviously, $FQ=GS+C$, with $C=R-P$, is a  LST  for
$u,v$. Let $K\in \Lambda_0$, $M=S(K+1)$  and  $C_1=C-RK$. Then,
$$
FQ-GM = G(S-M)+C = -GSK+C_1 +RK=C_1,
$$
and thus  we obtain a new LST  for every $K$. Choosing $K=1$, we
find that $FQ=2GS+C-R$. In this LST, $(Q,2S,C-R)$ is not
proportional to $(Q,S,C)$.
\end{proof}

\medskip

  As a direct consequence of the previous results we find that the
  LST relating two functionals $u,v\in \mathcal{H}\setminus{\rm
  Rat}$ are generated by the unique minimal one $FL_0=GM_0+C_0$,
  i.e., $\{(L,M,C)\in \Lambda_0\times \Lambda_0\times \Lambda \;|\;
  FL=GM+C\} = \Lambda (L_0,M_0,C_0)$.

\medskip

  Besides, according to Corollary \ref{MINIM}, Corollary \ref{CHAR} and Theorem  \ref{CHART}, to verify
  that there is a RM behind a LST, we only need to check that the
  minimal representation $FL=GM + C$ of a LST satisfies the
  symmetry conditions $(L^{*_p}, M^{*_p}, C^{*_p}) = \alpha(L,M,-C)$ for
  some $\alpha\in\C^*$, $p\in\Z$, i.e., $(L_*, M_*, C_*) \in{\rm{U}}(L,M,-C)$.

\section{General LST} \label{GENLST}
\medskip

  Theorem  \ref{CHART}  provides the conditions which characterize that a LST,
  $FL=GM+C$, in its minimal form, comes from a RM. In this section
  we will analyze what happens if we remove these symmetry conditions
  $(L_*, M_*, C_*) = {\rm U}(L,M,-C)$.

   Returning to Example 3.1, we can easily  check   that  the Laurent polynomials
  $L=1$, $M=(1-z^2)$, $C=0$, define the minimal form of the corresponding LST,
  but do not satisfy these symmetry conditions. This means that Example
  3.1 is a case of LST which does not come from a RM.

  Although the  functionals involved in Example 3.1 are both given
  by positive measures, we should remark that one of these measures
  is supported on a single point. The following example shows that
  there exist positive measures supported on infinitely many points
  which are related by a LST but not by a RM.

\medskip

\bex
 We consider the functional $u$ associated with the Lebesgue measure $\mu$
 supported on the arc $\Gamma=\{e^{i\theta} \;|\; \theta \in \left[{\pi\over 2}, {3\pi\over 2}\right]\}$.
 Is is easy to check that the corresponding moments are
$$
\mu_0=1;\qquad \mu_{2n}=0,\qquad \mu_{2n-1}={2\over \pi}{(-1)^n\over
2n-1}, \qquad n\geq 1.
$$

Therefore, the CS associated with $u$ is
$$ F(z)=1+{4\over
\pi}\sum_{n\geq 1} (-1)^n {z^{2n-1}\over 2n-1}, $$ which is summable
for $|z|<1$
$$
F(z)=1 + {2i\over \pi} \log\left[{1+iz\over 1-iz}\right],
$$
and has  real part
$$
\re F(re^{i\theta}) = 1-{2\over \pi}\pi
\arg\,[(1+ire^{i\theta})(1+ire^{-i\theta})].
$$

 Now, we perform the LST
 $$ G(z) = z
F(z)+\alpha,
$$
where we will choose $\alpha$ so that  $\re(G(z))>0$ for $|z|<1$.
This condition ensures that $G$ is the Carath\'{e}odory function of a functional $v$ given
by a positive measure $\nu$ supported on the unit circle.
A straightforward computation yields
$$
\begin{aligned}
\re G(r e^{i\theta}) &\geq \alpha + r \cos \theta \re F(re^{i\theta})\\
&=\alpha + r \cos \theta\left(1-{2\over \pi} \arg[(1+
re^{i\theta})(1+i e^{-i\theta})]\right)\geq \alpha-2r,
\end{aligned}
$$
which means that $G$ is a Carath\'eodory function for $\alpha\geq
2$. The radial limits of $G$ provide the weight of the corresponding measure $\nu$,
which is absolutely continuous because $G$ has the same
analyticity behavior as $F$. We obtain
$$
\nu'(\theta)= \alpha + 2 \cos\theta\,\chi_\Gamma(\theta) + {1\over
\pi} \sin\theta \log\left[{1+\sin \theta \over 1-\sin
\theta}\right].
$$
where $\chi_\Gamma$ is the characteristic function of the arc
$\Gamma$.

While $\mu$ is supported only on the arc $\Gamma$, the measure $\nu$
is supported on the whole unit circle $\T$. Consequently, although
$u, v$ are related by a LST, they can not be related by a RM.

Note that the Laurent polynomials in the minimal LST,
$G(z)=zF(z)+2$, do not satisfy the symmetry conditions
$(L_*,M_*, C_*) = {\rm U}(L,M,-C)$. \eex

 In the previous examples we have shown that the transformations in $\{\rm LST\} \setminus \{\rm RM\}$, in
contrast to RM, do not preserve the support of the absolutely
continuous part of the measure. We will refer to this as a {\sl wild
behaviour}. In what follows  we analyze the origin of this behaviour
in the connection $u \leftrightarrow v$.

\medskip

    We consider a general LST, i.e.,
without any conditions over the Laurent polynomials $L,M,C$, to find
out the roots of the wild behavior. Thus, we consider the LST
 \begin{equation}
 \label{GENLST1}
 FL=GM+C,\qquad L,M\in\Lambda_0,\qquad C\in\Lambda.
 \end{equation}

   Without loss of generality, we can choose $L=L^{*_p}$ for some
 $p\in\Z$,  multiplying the LST by a suitable  factor in $\Lambda_0$.
\medskip
\noindent From (\ref{GENLST1}), $ F_*L_*=G_*M_*+C_*$. Equivalently,
 \begin{equation}\label{GENLST2}
F_*L=F_*L^{*_p} =G_*M^{*_p} + C^{*_p}.
\end{equation}
 From (\ref{GENLST1}) and (\ref{GENLST2}),
$$
{F+F_*\over 2} L = {G+G_*\over 2}{M+M^{*_p}\over 2} - {G-G_*\over
2i}{M-M^{*_p}\over 2i} + {C+C^{*_p}\over 2}.
$$
Denoting
$$
M^+={M+M^{*_p}\over 2},\quad M^-={M-M^{*_p}\over 2i},\quad C^+=
{C+C^{*_p}\over 2},
$$ we have
\begin{equation}\label{GENLSTABREV}
{F+F_*\over 2} L={G+G_*\over 2} M^+ - {G-G_*\over 2i} M^-+C^+.
\end{equation}

   The LS  $\ds {G-G_*\over 2i}$ has an associated hermitian functional that we will denote by $\hat{v}$. Then,
  (\ref{GENLSTABREV})  can be written in terms of the functionals $u,v$ and
  $\hat{v}$ as
  \begin{equation}\label{GENLSTABREVFUNCT}
  uL_*=v(M^+)_* + \hat{v}(M^-)_*+ \leb (C^+)_*.
  \end{equation}
Nevertheless,
$$
(M^+)_* = {M_*+z^{-p}M\over 2} = z^{-p} {M+M^{*_p}\over 2} =
z^{-p}M^+,
$$
and,  analogously,  $(M^-)_*=z^{-p}M^-$ and $(C^+)_*= z^{-p}C^+$.
Multiplying Equation (\ref{GENLSTABREVFUNCT}) by $z^{p}$ gives
 \begin{equation}\label{GENLSTFINAL}
 uL= vM^+-\hat{v}M^- + \leb C^+,
 \end{equation}
 where $(L,M^+,M^-,C^+) = (L,M^+,M^-,C^+)^{*_p}$.
\medskip
  The functional $\hat{v}$   has the  associated LS
  $$
  {G-G_*\over 2i} = {1\over i}\left(\sum_{n\geq 1}\nu_{-n} z^n - \sum_{n\geq 1} \nu_n z^{-n}\right) = \sum_{n\neq 0} \hat{\nu}_{-n}z^n,
  $$
  with $\nu_n=v[z^n]$ and $\hat{\nu}_n= -i \nu_n=\hat{u}[z^n]$. Denoting by $\hat{G}$ the CS of $\hat{v}$, we have
$$
\hat{G} = 2\sum_{n\geq 1} \hat{\nu}_{-n} z^n = i (\nu_0 - G),
$$
$$
\hat{G}_* = 2\sum_{n\geq 1} \hat{\nu}_{n} z^{-n} = i (G_* -\nu_0),
$$
and, consequently,
$$
\hat{v}= \begin{cases}  &  i(v-\leb \nu_0)\quad ({\rm in}\;
\PP),\\& i(\leb \nu_0-v)\quad ({\rm in}\;\; \PP_*).\\
\end{cases}
$$

 Subtracting (\ref{GENLST1}) and (\ref{GENLST2}),
 $$
 {F-F_*\over 2i} L = {G+G_*\over 2}{M-M^{*_p}\over 2i} + {G-G_*\over 2}{M+M^{*_p}\over 2i}+ {C-C^{*_p}\over 2i},
 $$
 i.e.,
 \begin{equation}\label{GENLSTFINALHAT}
 \hat{u} L = vM^- + \hat{v} M^+ + \leb\;C^-,
 \end{equation}
 where $C^-=(C-C^{*_p})/2i$  and
 $$
 \hat{u}= \begin{cases}& i(u-\leb \mu_0)\quad ({\rm in}\;
\PP),\\ & i(\leb \mu_0-u)\quad ({\rm in}\;
 \PP_*),\\ \end{cases}
 $$
 with $\mu_n= u[z^n]$. The expressions (\ref{GENLSTFINAL}) and (\ref{GENLSTFINALHAT}) can be written in
 matrix form as
 $$
 (u, \hat{u})L = (v,\hat{v}) \begin{pmatrix} M^+ &
M^-\cr -M^-&M^+\end{pmatrix} + \leb (C^+,C^-).
$$

\medskip

 The wild behaviour of the LST is originated by the presence of the
 functional $\hat{v}$ and the Lebesgue functional in Equation \ref{GENLSTFINAL}.
 In other words, the relation between $u$ and $v$ fails to be a RM
 only due to the presence of the polynomial coefficients $M^-$ and
 $C^+$ in Equation \ref{GENLSTFINAL}.
 Therefore, a LST becomes a RM when $M^-=C^+=0$. This leads to
 the symmetry conditions which we already know that characterize those
 LST coming from a RM.

  \medskip

\section{Generators of  LST} \label{GENERATORSLST}
\medskip
  In this section we analyze another aspect of the LST, namely, we will show how to generate
  them by composition of certain elementary LST that we will call generators.

  First of all, without loss of generality, we will write any LST as
$$
FA = GB + C,\qquad A,B,C\in \PP.
$$
  Then, we define the class $(r,s,t)$ by the conditions $\deg A=r$, $\deg B=s$, $\deg C=t$
  and $(A(0),B(0),C(0))\neq (0,0,0)$. If $C=0$ we will say that the corresponding
  LST  belongs to the $(r,s)$ class.

\medskip

  Note that the (0,0) class is generated by composition of
 LST belonging to
 the (0,0,0) class. The (1,0) and (0,1) classes generate by
 composition all the $(r,s)$ classes  with $(r,s)\not=(0,0)$. Analogously, the
(1,0,0) and (0,1,0) classes generate the $(r,s,0)$ classes with $(r,s)
\neq (0,0)$. At the same time, the $(1,0,0)$ and $(0,1,0)$ classes
 are generated by the $(0,0,0), (1,0)$ and $(0,1)$ ones: any $(1,0,0)$ class LST
 such as $F(\alpha_0+\alpha_1 z)= G \beta+c$ is the composition of
 $ F(\alpha_0+\alpha_1 z)=\tilde{F}\alpha_0$ and $\tilde{F}\alpha_0= G \beta+ c$.

\medskip

  Consider now the elementary classes $(0,0,0)$, $(1,0)$ and $(0,1)$. We will
  show that they generate all the classes $(r,s,t)$ with
  $r,s,t\geq 0$. For this purpose,  it will be enough to prove that every
  $(0,0,t )$ class LST, with $t\geq 1$, can be transformed into a
  $(0,0,t_1)$ class LST, with $t_1 < t$, by means of elementary LST.
  This is because any LST in the $(r,s,t)$ class can be reduced to a
  LST in the class $(0,0,t')$ or $(0,0)$ by $(1,0,0)$ and $(0,0,0)$
  class LST, as it is easy to check.

   Let
   \begin{equation}\label{GNTR} FA = GB + C
   \end{equation}
   be a   LST  of class $(r,s,t)$. Let
   $A_1=A/A_0$ where $A_0$ is a degree one divisor of $A$. The
   composition of $FA_0=\tilde{F}+c$ and  $\tilde{F}A_1 = G B+
   C-cA_1$ gives (\ref{GNTR}), the first LST being in the $(1,0,0)$ class, and the second
   one in the $(r-1,0,t')$ class  with $t'=\deg(C-cA)\leq \max\{r-1,t\}$.

  Now, we will see that every  $(0,0,t)$ class LST
  $$
 F   \alpha = G \beta + C, \qquad t\geq 1,
  $$
  is generated by the $(0,0,0)$, $(1,0)$ and $(0,1)$ classes.

Let us consider the  $(1,0,0)$ and $(1,0)$ class transformations
 $$
\tilde{F} z = {F-\mu_0\over 2\mu_{-k}},\qquad  \tilde{G} z
={G-\nu_0\over
 2\nu_{-j}},
 $$
 where $\mu_{-k}$ and $\nu_{-j}$ are the first  non zero coefficients
 of $F$ and $G$ with indices $k,j \geq 1$.
 From the equality $F \alpha  =  G \beta +C$ we have
 $$
 \alpha\mu_0=\beta\nu_0 + C(0),
 $$
 and from
 $$
 (\tilde{F}2\mu_{-k}z  + \mu_0)\alpha =
(\tilde{G}2\nu_{-j}z+\nu_0) \beta+C
 $$
 it follows that
 $$
\tilde{F} 2\alpha \mu_{-k}z  + \alpha\mu_0  =
 \tilde{G} 2\beta \nu_{-j}z+\beta\nu_0 +C_1,
 $$
 with $C_1(z)=(C(z)-C(0))/z$ such that $\deg C_1\leq t-1$.

 This proves that the elementary LST classes
 generate by composition  all the LST.

 \medskip

 Finally, we will analyze the relation between the functionals related by the
elementary classes.

 A general $(0,0,0)$ class LST has the form
$$
 F \alpha=G\beta  +c,
$$
where, without loss of generality, we can assume that $\alpha\in\R$, i.e. $A=A^*$. In
such a case
$$
u \alpha  =  v \re(\beta) - \hat{v}  \im (\beta) + \re(c),
$$
in other words,
$$
\begin{aligned}
 u \alpha&=  v \re(\beta)+ v i \im (\beta)  -
v _0 i \im (\beta) +
\re(c)\quad ({\rm in}\;\PP),\\
u \alpha &= v \re(\beta) - v i \im(\beta) +  v_0 i \im (\beta)+
\re(c)\quad ({\rm in}\;\PP_*).
\end{aligned}
$$
Equivalently,
$$
\begin{aligned}
 u \alpha &= v \beta -v_0i  \im (\beta)  +\re(c)\quad({\rm in}\;\PP),\\
 u \alpha &= v {\bar \beta}+ v_0 i \im (\beta) +\re (c)\quad({\rm in}\;\PP_*),
\end{aligned}
$$
which describe completely the functional modification.
\medskip

 As for the $(0,1)$ class LST
$$
\hskip 25pt  F \alpha= G B, \hskip 25pt B=\beta_0+\beta_1z, \hskip
10pt \alpha \in\R^*,
$$
we have that $\alpha \mu_0=\beta_0\nu_0$, thus $\beta_0\in\R$. Since
$$
u \alpha = v{B+B^*\over 2} - \hat{v} {B-B^*\over 2i},
$$
we find that
$$
\begin{aligned}
 u \alpha&=   v B - \leb i {\beta_1 z - \overline{\beta}_1 z^{-1}\over 2i}\nu_0
\qquad\;({\rm in}\; \PP),\\
 u \alpha&=   v B_*  -  \leb i {\beta_1 z + \overline{\beta}_1 z^{-1}\over 2i}\nu_0
\qquad({\rm in}\; \PP_*).
\end{aligned}
$$

    Analogously,  for the $(1,0)$ class LST
 $$
  F (\alpha_0+\alpha_1 z)= G  \beta, \qquad \beta\in\R^*,
 $$
 we conclude that $\alpha_0\in\R$ and
 $$
 \begin{aligned}
 v \beta&=
 u A- \leb i {\alpha_1 z - \overline{\alpha}_1 z^{-1}\over 2i}
 \mu_0\;\;\;\;\qquad ({\rm in}\;\;\PP),\\
 v \beta&=  u A_* + \leb i {\alpha_1 z - \overline{\alpha}_1
 z^{-1}\over 2i} \mu_0\ \;   \qquad ({\rm in}\; \PP_*),\\
\end{aligned}
 $$
 where $A=\alpha_0+\alpha_1 z$.

 \bigskip

 \noindent{\bf Acknowledgements}

\medskip
This work was partially supported by the research
 projects MTM2008-06689-C02-01 and MTM2011-28952-C02-01 from the Ministry of  Science and
Innovation of Spain and the European Regional Development Fund
(ERDF), and by Project E-64 of Diputaci\'on General de Arag\'on
(Spain).

\end{document}